\documentclass[12pt, reqno]{amsart}
\usepackage{ amsmath,amsthm, amscd, amsfonts, amssymb, graphicx, color}
\usepackage[bookmarksnumbered, colorlinks, plainpages]{hyperref}
\newtheorem{thm}{Theorem}[section]
\newtheorem{cor}[thm]{Corollary}
\newtheorem{lem}[thm]{Lemma}

\newtheorem{exam}[thm]{Example}
\numberwithin{equation}{section}

\begin{document}

\title{Jacobson's Lemma for the generalized n-strongly Drazin inverse}

\author{Huanyin Chen}
\author{Marjan Sheibani}
\address{
Department of Mathematics\\ Hangzhou Normal University\\ Hang -zhou, China}
\email{<huanyinchen@aliyun.com>}
\address{
Women's University of Semnan (Farzanegan), Semnan, Iran}
\email{<sheibani@fgusem.ac.ir>}


\begin{abstract}
Let $n\in {\Bbb N}$. An element $a\in R$ has generalized n-strongly Drazin inverse if there exists $x\in R$ such that $xax=x, x\in comm^2(a), a^n-ax\in R^{qnil}.$
For any $a,b\in R$, we prove that $1-ab$ has generalized n-strongly Drazin inverse if and only if $1-ba$ has generalized n-strongly Drazin inverse. Extensions in Banach algebra are also obtained.
\end{abstract}

\maketitle

\section{Introduction}

Let $R$ be an associative ring with an identity. The commutant of $a\in R$ is defined by $comm(a)=\{x\in
R~|~xa=ax\}$. The double commutant of $a\in R$ is defined
by $comm^2(a)=\{x\in R~|~xy=yx~\mbox{for all}~y\in comm(a)\}$. Set $R^{qnil}=\{a\in R~|~1+ax\in U(R)~\mbox{for
every}~x\in comm(a)\}$. An element $a\in R$ is quasinilpotent if
$a\in \grave{R^{qnil}}$. For a Banach algebra $\mathcal{A}$ it is well known
 that $$a\in \mathcal{A}^{qnil}\Leftrightarrow
\lim\limits_{n\to\infty}\parallel a^n\parallel^{\frac{1}{n}}=0.$$
An element $a$ in $R$ is said to have generalized Drazin inverse if there exists $x\in R$ such that $$x=xax, x\in comm^2(a), a-a^2x\in R^{qnil}.$$  The preceding $x$ is unique, if such element exists. As usual,
it will be denoted by $a^{d}$, and called the generalized Drazin inverse of $a$. Generalized Drazin inverse is extensively studied in matrix theory and Banach algebra (see~\cite{B, J, LZ} and~\cite{ZC}).

Let $n\in {\Bbb N}$. Following Mosic (see~\cite{M}), an element $a\in R$ has generalized n-strongly Drazin inverse if there exists $x\in R$ such that $$xax=x, x\in comm^2(a), a^n-ax\in R^{qnil}.$$ In this case, $x$ is a generalized n-strongly Drazin inverse of $a$.

The topic for generalized n-strongly Drazin inverse was studied by many authors. In ~\cite{CS, G, Mo} and ~\cite{W}, Chen et al. studied such generalized inverse for $n=1$. In~\cite{CM, CM1} and~\cite{CM2}, the authors investigated this generalized inverse for $n=2$. For a Banach algebra $\mathcal{A}$, we proved that
$a\in \mathcal{A}$ has generalized 2-strongly Drazin inverse if and only if $a-a^3\in \mathcal{A}^{qnil}$ if and only if $a$ is the sum of a tripotent and a quasinilpotent that commute (see~\cite[Theorem 2.4 and Theorem 2.7]{CM2}). In~\cite{ZM}, Zou considered the generalized n-strongly Drazin inverse in terms of nilpotents.

Jacobson's Lemma was initially a statement for the classical inverse in a ring. It claims that $1-ab$ is invertible if and only if $1-ba$ is invertible.
It was then extended to inner inverses, group inverses and Drazin inverse. More recently, it was generalized to generalized Drazin inverses. That is, $1-ab$ has generalized Drazin inverse if and only if $1-ba$ has generalized Drazin inrverse (see~\cite[Theorem 2.3]{ZC}).

In~\cite[Theorem 2.4]{M}, Mosic investigated Jacobson's Lemma for generalized 1-strongly Drazin inverse. The purpose of this paper is to extend Jacobson's Lemma from generalized Drazin inverse to generalized n-strongly Drazin inverse. We shall prove that $1-ab$ has generalized n-strongly Drazin inverse if and only if $1-ba$ has generalized n-strongly Drazin inverse. So as to study the common spectral properties for bounded linear operators, certain generalizations of Jacobson's Lemma are investigated (see~\cite{MZ, Mos, Y1, Y, Z} and~\cite{ZZ}). In the last section, two extensions for the generalized n-strongly Drazin inverse in a Banach algebra are also obtained.

Throughout the paper, all rings are associative with an identity and all Banach algebras are complex. We use $n$ to denote a fixed natural number. $U(R)$ stands for the set of all units in a ring $R$.

\section{Jacobson's Lemma}

The aim of this section is to extend Theorem 2.3 in~\cite{ZC} to generalized n-strongly Drazin inverse. We begin with

\begin{lem} Let $R$ be a ring, and let $a,b\in R$. Then $ab\in R^{qnil}$ if and
only if $ba\in R^{qnil}$.\end{lem} \begin{proof}  See~\cite[Lemma 2.2]{LZ}.\end{proof}

\begin{lem} Let $R$ be a ring, $a\in R$ and $n\in {\Bbb N}$. Then the following are equivalent:\end{lem}
\begin{enumerate}
\item [(1)]{\it $a$ has generalized n-strongly Drazin inverse.}
\vspace{-.5mm}
\item [(2)]{\it There exists $e^2=e\in R$ such
that
$$e\in comm^2(a), a^n-e\in R^{qnil}.$$}
\end{enumerate}
\begin{proof} See ~\cite[Theorem 2.1]{M}.\end{proof}

Now we come to state the main result.

\begin{thm} Let $R$ be a ring, and let $a,b\in R$. Then $1-ab$ has generalized n-strongly Drazin inverse if and only if $1-ba$ has generalized n-strongly Drazin inverse.\end{thm}
\begin{proof} $\Longrightarrow$ Let $\alpha=(1-ab)^n$ and . Then $$\begin{array}{lll}
\alpha&=&\sum\limits_{i=0}^{n}(-1)^{i}\left(
\begin{array}{c}
n\\
i
\end{array}
\right)(ab)^{i}\\
&=&1-a\sum\limits_{i=1}^{n}(-1)^{i-1}\left(
\begin{array}{c}
n\\
i
\end{array}
\right)b(ab)^{i-1}\\
&=&1-ac,
\end{array}$$ where $c=\sum\limits_{i=1}^{n}(-1)^{i-1}\left(
\begin{array}{c}
n\\
i
\end{array}
\right)b(ab)^{i-1}$.

Let $\beta=(1-ba)^n$. Then
$$\begin{array}{lll}
\beta&=&\sum\limits_{i=0}^{n}(-1)^{i}\left(
\begin{array}{c}
n\\
i
\end{array}
\right)(ba)^{i}\\
&=&1-\big(\sum\limits_{i=1}^{n}(-1)^{i-1}\left(
\begin{array}{c}
n\\
i
\end{array}
\right)(ba)^{i-1}b\big)a\\
&=&1-ca.
\end{array}$$
In view of Lemma 2.2, there exists $p^2=p\in R$ such
that
$$p\in comm^2(1-ac), \alpha-p\in R^{qnil}.$$
Then $1-(1-p)\alpha=1-(1-p)(\alpha -p)\in U(R)$. Set $q=c(1-p)(1-(1-p)\alpha)^{-1}a$. Then
$$\begin{array}{lll}
q^2&=&c(1-p)(1-(1-p)\alpha)^{-1}(ac)(1-p)(1-(1-p)\alpha)^{-1}a\\
&=&c(1-(1-p)\alpha)^{-1}((1-p)ac)(1-(1-p)\alpha)^{-1}(1-p)a\\
&=&c(1-(1-p)\alpha)^{-1}(1-p)(1-\alpha)(1-(1-p)\alpha)^{-1}(1-p)a\\
&=&c(1-(1-p)\alpha)^{-1}(1-p)(1-(1-p)\alpha)(\alpha+1-p)^{-1}a\\
&=&c(1-p)(1-(1-p)\alpha)^{-1}a\\
&=&q.
\end{array}$$
Clearly, $(1-ca)c=c(1-ac)$, and so we have
$$\begin{array}{lll}
(1-ca)q&=&(1-ca)c(1-p)(1-(1-p)\alpha)^{-1}a\\
&=&c(1-p)(1-(1-p)\alpha)^{-1}a(1-ca)\\
&=&q(1-ca),
\end{array}$$
i.e., $caq=qca$. We claim that $q\in comm^2(1-ca)$. Let $y\in R$ be such that $y(1-ca)=(1-ca)y$. Then $yca=cay$. This implies that
$(ayc)ac=ac(ayc)$. We infer that $(ayc)(1-ac)=(1-ac)(ayc)$, and so $(ayc)\alpha=\alpha (ayc)$. As $p\in comm^2(1-ac)$, we get $(ayc)p=p(ayc)$.

Thus, we have $$\begin{array}{ll}
&(ayc)(1-p)(1-(1-p)\alpha)^{-1}\\
=&(1-p)(ayc)(1-(1-p)\alpha)^{-1}\\
=&(1-p)(1-(1-p)\alpha)^{-1}(ayc).
\end{array}$$
Hence, $$\begin{array}{lll}
(cay)q&=&cayc(1-p)(1-(1-p)\alpha)^{-1}a\\
&=&c(1-p)(1-(1-p)\alpha)^{-1}ayca\\
&=&c(1-p)(1-(1-p)\alpha)^{-1}a(cay)\\
&=&q(cay)\\
&=&q(yca),
\end{array}$$ and so $(ca)yq=qy(ca).$
Then we check that $$\begin{array}{lll}
yq(1-(1-ca)q)&=&yq(1-q(1-ca))\\
&=&yqca\\
&=&ycaq\\
&=&cayq\\
&=&qyca.
\end{array}$$ Multiplying the above by $q$ on the right side yields
$$yq(1-(1-ca)q)=qyq(1-(1-ca)q).$$ Obviously, $1-(1-ca)q=1-(1-ca)c(1-p)(1-(1-p)\alpha)^{-1}a$.
Since $$\begin{array}{ll}
&1-a(1-ca)c(1-p)(1-(1-p)\alpha)^{-1}\\
=&1-ac(1-ac)(1-p)(1-(1-p)\alpha)^{-1}\\
=&1-(\alpha-\alpha^2)(1-p)(1-(1-p)\alpha)^{-1}\\
=&1-(1-\alpha)\alpha (1-p)(1-(1-p)\alpha)^{-1}\\
\in& U(R),
\end{array}$$ by using Jacobson's Lemma (~\cite[Corollary 2.5]{Y}), we see that $1-(1-ca)q\in U(R)$. This implies that $yq=qyq$.
Hence $$\begin{array}{lll}
(1-(1-ca)q)qy&=&caqy\\
&=&qyca\\
&=&cayq\\
&=&caqyq\\
&=&(1-(1-ca)q)qyq,
\end{array}$$
and so $qy=qyq$. Therefore $yq=qyq=qy$, and then $q\in comm^2(1-ca)$.

Write $r=\big((1-p)(1-(1-p)\alpha)^{-1}-1\big)a$. Then $$\begin{array}{lll}
rc&=&\big((1-p)(1-(1-p)\alpha)^{-1}-1\big)ac\\
&=&\big((1-p)(1-(1-p)\alpha)^{-1}-1\big)(1-\alpha)\\
&=&(1-(1-p)\alpha)^{-1}(1-p)(1-\alpha)-(1-\alpha)\\
&=&1-p-1+\alpha\\
&=&\alpha-p,
\end{array}$$ and so
$rc\in R^{qnil}$. In light of Lemma 2.1, $cr\in R^{qnil}$.

On the other hand, $$\begin{array}{lll}
\beta+q&=&1-ca+c(1-p)(1-(1-p)\alpha)^{-1}a\\
&=&1-c\big(1-(1-p)(1-(1-p)\alpha)^{-1}\big)a\\
&=&1+cr,
\end{array}$$ and then $\beta-(1-q)=cr\in R^{qnil}.$
Obviously, $1-q=(1-q)^2\in comm^2(1-ca)$.
Accordingly, $1-ba$ has generalized n-strongly Drazin inverse.

$\Longleftarrow$ This is symmetric.\end{proof}

\begin{cor} Let $R$ be a ring, let $m\in {\Bbb N}$, and let $a,b\in R$. Then $(1-ab)^m$ generalized n-strongly Drazin inverse if and only if $(1-ba)^m$ has generalized n-strongly Drazin inverse.
\end{cor}
\begin{proof} $\Longrightarrow$ One easily checks that $b(1-(ab)^k)=(1-(ba)^k)b$ for all $k\in {\Bbb N}$. Then
$$\begin{array}{lll}
1-(1-ab)^m&=&ab(1+(1-ab)+\cdots +(1-ab)^{m-1})\\
&=&a(b+b(1-ab)+\cdots +b(1-ab)^{m-1})\\
&=&a(b+(1-ba)b+\cdots +(1-ba)^{m-1}b)\\
&=&a(1+(1-ba)+\cdots +(1-ba)^{m-1})b.
\end{array}$$
Hence, $$(1-ab)^m=1-a(1+(1-ba)+\cdots +(1-ba)^{m-1})b.$$
Likewise,  $$\begin{array}{lll}
(1-ba)^m&=&1-b(1+(1-ab)+\cdots +(1-ab)^{m-1})a\\
&=&1-(1+(1-ba)+\cdots +(1-ba)^{m-1})ba.
\end{array}$$
In view of Theorem 2.3, $(1-ba)^m$ has generalized n-strongly Drazin inverse, as desired.

$\Longleftarrow$ This is symmetric.\end{proof}

\begin{cor} Let $R$ be a ring, and let $A\in M_{k\times l}(R), B\in M_{l\times k}(R)$. Then $I_k+AB\in M_k(R)$ has generalized n-strongly Drazin inverse if and only if $I_l+BA\in M_l(R)$ has generalized n-strongly Drazin inverse.
\end{cor}
\begin{proof} Let $m=k+l$. Set $$C=
\left(
\begin{array}{cc}
0&0\\
A&0
\end{array}
\right), D=
\left(
\begin{array}{cc}
0&B\\
0&0
\end{array}
\right)\in M_{m\times m}(R).$$
Then we observe that
$$I_m+CD=\left(
\begin{array}{cc}
I_l&0\\
0&I_k+AB
\end{array}
\right), I_m+DC=\left(
\begin{array}{cc}
I_l+BA&0\\
0&I_k
\end{array}
\right).$$
In light of Theorem 2.3, $I_k+CD\in M_{k\times k}(R)$ has generalized n-strongly Drazin inverse if and only if so has $I_l+DC\in M_{l\times l}(R)$. This completes the proof.\end{proof}

\section{Extensions in Banach algebra}

Jacobson's Lemma is of interest in spectral theory in Banach algebras. Over the years, certain extensions of Jacobson's Lemma were founded for many operator properties (see~\cite{M, MZ, Z} and ~\cite{ZZ}). The goal of this section is to generalize Theorem 2.3 to wider cases in a Banach algebra. Recall that $a\in \mathcal{A}$ has generalized strongly Drazin inverse if there exists $x\in R$ such that $xax=x, x\in comm(a), a-ax\in \mathcal{A}^{qnil}.$ The following lemma is crucial.

\begin{lem} Let $\mathcal{A}$ be a Banach algebra, and let $a\in \mathcal{A}$. Then the following are equivalent:
\end{lem}
\begin{enumerate}
\item [(1)]{\it $a$ has generalized n-strongly Drazin inverse.}
\vspace{-.5mm}
\item [(2)]{\it $a^n$ has generalized strongly Drazin inverse.}
\end{enumerate}
\begin{proof} $(1)\Rightarrow (2)$ In view of~\cite[Lemma 2.1]{M}, $a\in \mathcal{A}$ has generalized Drazin inverse. Furthermore, we have $a^n-aa^d\in \mathcal{A}^{qnil}$.
Hence, $a^n-a^n(a^d)^n=a^n-(aa^d)^n\in \mathcal{A}^{qnil}$. In light of~\cite[Theorem 2.7]{J}, $a^n\in \mathcal{A}$ has generalized Drazin inverse and $(a^n)^d=(a^d)^n$.
Therefore $a^n-a^n(a^n)^d\in \mathcal{A}$, and so $a^n\in \mathcal{A}$ has generalized strongly Drazin inverse.

$(2)\Rightarrow (1)$ Obviously, $a^n\in \mathcal{A}$ has generalized Drazin inverse. It follows by ~\cite[Theorem 2.7]{J} that $a\in \mathcal{A}$ has generalized Drazin inverse and $(a^d)^n=(a^n)^d$. Since $a^n-aa^d=a^n-a^n(a^d)^n=a^n-a^n(a^n)^d\in \mathcal{A}^{qnil}$, we see that $a\in \mathcal{A}$ has generalized n-strongly Drazin inverse.\end{proof}

We are ready to prove:

\begin{thm} Let $a,b,c,d\in \mathcal{A}$ satisfy $acd=dbd$ and $dba=aca$. Then $1-ac$ has generalized n-strongly Drazin inverse if and only if $1-bd$ has generalized n-strongly Drazin inverse.\end{thm}
\begin{proof} $\Longrightarrow$ Let $\alpha=(1-ac)^n$. Then $$\begin{array}{lll}
\alpha&=&\sum\limits_{i=0}^{n}(-1)^{i}\left(
\begin{array}{c}
n\\
i
\end{array}
\right)(ac)^{i}\\
&=&1-a\sum\limits_{i=1}^{n}(-1)^{i-1}\left(
\begin{array}{c}
n\\
i
\end{array}
\right)c(ac)^{i-1}\\
&=&1-ac',
\end{array}$$ where $c'=\sum\limits_{i=1}^{n}(-1)^{i-1}\left(
\begin{array}{c}
n\\
i
\end{array}
\right)c(ac)^{i-1}$.

Let $\beta=(1-bd)^n$. Then
$$\begin{array}{lll}
\beta&=&\sum\limits_{i=0}^{n}(-1)^{i}(bd)^{i}\\
&=&1-\big(\sum\limits_{i=1}^{n}(-1)^{i-1}\left(
\begin{array}{c}
n\\
i
\end{array}
\right)(bd)^{i-1}b\big)d\\
&=&1-b'd,
\end{array}$$ where $b'=\sum\limits_{i=1}^{n}(-1)^{i-1}\left(
\begin{array}{c}
n\\
i
\end{array}
\right)(bd)^{i-1}b.$
We are easy to verify that $$\begin{array}{lll}
ac'd&=&a\sum\limits_{i=1}^{n}(-1)^{i-1}\left(
\begin{array}{c}
n\\
i
\end{array}
\right)c(ac)^{i-1}d\\
&=&\sum\limits_{i=1}^{n}(-1)^{i-1}\left(
\begin{array}{c}
n\\
i
\end{array}
\right)(ac)^{i}d\\
&=&\sum\limits_{i=1}^{n}(-1)^{i-1}(db)^{i}d\\
&=&d\sum\limits_{i=1}^{n}(-1)^{i-1}\left(
\begin{array}{c}
n\\
i
\end{array}
\right)(bd)^{i-1}bd\\
&=&db'd;
\end{array}$$
$$\begin{array}{lll}
db'a&=&d\sum\limits_{i=1}^{n}(-1)^{i-1}\left(
\begin{array}{c}
n\\
i
\end{array}
\right)(bd)^{i-1}ba\\
&=&\sum\limits_{i=1}^{n}(-1)^{i-1}\left(
\begin{array}{c}
n\\
i
\end{array}
\right)d(bd)^{i-1}ba\\
&=&\sum\limits_{i=1}^{n}(-1)^{i-1}\left(
\begin{array}{c}
n\\
i
\end{array}
\right)(db)^{i}a\\
&=&\sum\limits_{i=1}^{n}(-1)^{i-1}(ac)^{i}a\\
&=&a\sum\limits_{i=1}^{n}(-1)^{i-1}\left(
\begin{array}{c}
n\\
i
\end{array}
\right)c(ac)^{i-1}a\\
&=&ac'a.
\end{array}$$
Then we have $ac'd=db'd, db'a=ac'a.$
By virtue of Lemma 3.1,
$\alpha$ has generalized strongly Drazin inverse. In light of ~\cite[Theorem 2.4]{ZC}, $\beta\in R$ has strongly generalized Drazin inverse.
By using Lemma 3.1 again, $1-bd$ has generalized n-strongly Drazin inverse, as asserted.

$\Longleftarrow$ This is symmetric.\end{proof}

\begin{cor} Let $a,b,c\in \mathcal{A}$ satisfy $aba=aca$. Then $1-ac$ has generalized n-strongly Drazin inverse if and only if $1-ba$ has generalized n-strongly Drazin inverse.
\end{cor}
\begin{proof} By choosing $d=a$, we complete the proof by Theorem 3.2.\end{proof}

\begin{thm} Let $a,b,c,d\in \mathcal{A}$ satisfy $acd=dbd$ and $bdb=bac$. If $1-ac$ has generalized n-strongly Drazin inverse, then $1-bd$ has generalized n-strongly Drazin inverse.\end{thm}
\begin{proof} Let $\alpha=(1-ac)^n$ and $\beta=(1-bd)^n$. As in the proof of Theorem 3.2, we have
$$\alpha=1-ac', \beta=1-b'd,$$ where $$c'=\sum\limits_{i=1}^{n}(-1)^{i-1}\left(
\begin{array}{c}
n\\
i
\end{array}
\right)c(ac)^{i-1}, b'=\sum\limits_{i=1}^{n}(-1)^{i-1}\left(
\begin{array}{c}
n\\
i
\end{array}
\right)(bd)^{i-1}b.$$
Also we easily check that $ac'd=db'd.$ Furthermore, we verify
$$\begin{array}{ll}
&b'db'\\
=&\big(\sum\limits_{i=1}^{n}(-1)^{i-1}\left(
\begin{array}{c}
n\\
i
\end{array}
\right)(bd)^{i-1}b\big)d\big(\sum\limits_{i=1}^{n}(-1)^{i-1}\left(
\begin{array}{c}
n\\
i
\end{array}
\right)(bd)^{i-1}b\big)\\
=&\big(\sum\limits_{i=1}^{n}(-1)^{i-1}\left(
\begin{array}{c}
n\\
i
\end{array}
\right)(bd)^{i-1}\big)\big(\sum\limits_{i=1}^{n}(-1)^{i-1}\left(
\begin{array}{c}
n\\
i
\end{array}
\right)(bd)^{i}b\big)\\
=&\big(\sum\limits_{i=1}^{n}(-1)^{i-1}\left(
\begin{array}{c}
n\\
i
\end{array}
\right)(bd)^{i-1}\big)\big(\sum\limits_{i=1}^{n}(-1)^{i-1}\left(
\begin{array}{c}
n\\
i
\end{array}
\right)b(ac)^{i}\big)\\
=&\big(\sum\limits_{i=1}^{n}(-1)^{i-1}\left(
\begin{array}{c}
n\\
i
\end{array}
\right)(bd)^{i-1}b\big)a\big(\sum\limits_{i=1}^{n}(-1)^{i-1}\left(
\begin{array}{c}
n\\
i
\end{array}
\right)c(ac)^{i-1}\big)\\
=&b'ac'.
\end{array}$$
Hence $ac'd=db'd$ and $b'db'=b'ac'.$
In view of~\cite[Lemma 2.1]{ZC}, $\alpha\in R$ has generalized strongly Drazin inverse.
By virtue of~\cite[Theorem 2.4]{ZC}, $\beta\in R$ has strongly generalized Drazin inverse.
According to Lemma 3.1, $1-bd$ has generalized n-strongly Drazin inverse, as asserted.

$\Longleftarrow$ This is symmetric.\end{proof}

Finally, we present an illustrating example which can be derived from Theorem 3.2, but not from Theorem 3.4.

\begin{exam} Let $$\begin{array}{c}
a=\left(
\begin{array}{cccc}
0&0&1&1\\
0&0&0&1\\
0&0&0&0\\
0&0&0&0
\end{array}
\right), b=c=\left(
\begin{array}{cccc}
1&0&1&0\\
0&1&0&1\\
1&1&1&0\\
0&0&0&1
\end{array}
\right),\\
d=\left(
\begin{array}{cccc}
1&1&0&1\\
0&0&0&0\\
0&0&0&0\\
0&0&0&0
\end{array}
\right)\in M_4({\Bbb C}).
\end{array}$$ Then $acd=dbd$ and $dba=aca$, while $bdb\neq bac$. In this case, $$\begin{array}{c}
I_4-ac=\left(
\begin{array}{cccc}
0&-1&-1&-1\\
0&1&0&-1\\
0&0&1&0\\
0&0&0&1
\end{array}
\right),\\
I_4-bd=\left(
\begin{array}{cccc}
0&-1&0&-1\\
0&1&0&0\\
-1&-1&1&-1\\
0&0&0&1
\end{array}
\right)
\end{array}$$ both have generalized 1-strongly Drazin inverses.\end{exam}

\end{document}